\documentclass{amsart}
\usepackage{amsmath,amssymb,amsthm, amsbsy,amsfonts,mathtools, graphicx, tabularx, latexsym, float, color} 
\usepackage[utf8x]{inputenc}
\usepackage{ucs}
\usepackage{scrextend}	
\usepackage{varioref}	
\usepackage{hyperref}
\usepackage{cleveref}

\usepackage{abstract}

\theoremstyle{thm}
\newtheorem{Theorem}{Theorem}[section]
\newtheorem{mainthm}{Theorem}
\newtheorem*{mainthmA}{Theorem A}
\newtheorem*{mainthmB}{Theorem B}
\newtheorem{Lemma}[Theorem]{Lemma}

\newtheorem{Definition}[Theorem]{Definition}
\newtheorem{defprop}[Theorem]{Definition/Proposition}
\newtheorem{Property}[Theorem]{Property}
\newtheorem{remark}[Theorem]{Remark}
\newtheorem{Example}[Theorem]{Example}

\numberwithin{equation}{section}

\makeatletter
\g@addto@macro\bfseries{\boldmath}
\makeatother

\newcommand{\Aff}{\mathbb{A}}	
\newcommand{\E}{\mathbb{E}}	
\newcommand{\R}{\mathbb{R}} 
\newcommand{\Sp}{\mathbb{S}} 
\newcommand{\Z}{\mathbb{Z}} 

\newcommand{\A}{\mathcal{A}} 
\newcommand{\RS}{\Phi} 

\renewcommand{\a}{\alpha}	
\newcommand{\D}{\Delta}	
\newcommand{\g}{\gamma}	
\renewcommand{\S}{\Sigma}	

\newcommand{\ch}[1]{\mathrm{Ch}(#1)} 

\newcommand{\id}{\mathrm{id}} 
\renewcommand{\int}[1]{\mathrm{int} (#1)}
\newcommand{\rk}{\mathrm{rk}}

\newcommand{\st}[2]{\mathrm{st}_{#1}(#2)} 

\newcommand{\set}[1]{\left\{#1\right\}} 
\newcommand{\lr}[1]{\left\langle#1\right\rangle} 

\newcommand{\fo}{\leq} 

\newcommand{\isom}{\cong}

\newcommand\restr[2]{{
		\left.\kern-\nulldelimiterspace 
		#1 
		\vphantom{\big|} 
		\right|_{#2} 
}}
\renewcommand{\mod}[2]{{\raisebox{0.0em}{$#1$}\left/\raisebox{-.25em}{$#2$}\right.}}

\begin{document}

%
%
%
%
%
%
%
%
%

\title[A structure theorem for euclidean buildings]{A structure theorem for euclidean buildings}

\author[Petra Schwer]{Petra Schwer}
\author{David Weniger}
\address{%
Petra Schwer\\
Otto-von-Guericke-Universität Magdeburg\\
Universitätsplatz 2\\
39106 Magdeburg\\
Germany}

\email{petra.schwer@ovgu.de, david.weniger@web.de}

\thanks{We would like to thank the anonymous referee for many helpful comments and suggestions. The first author would like to thank Misha Kapovich for helpful discussions. }



\date{January 12th, 2019}

\maketitle

\begin{abstract}
We prove an affine analog of Scharlau's reduction theorem for spherical buildings. To be a bit more precise let $X$ be a euclidean building with spherical building $\partial X$ at infinity. Then there exists a euclidean building $\bar X$ such that $X$ splits as a product of $\bar X$ with some euclidean $k$-space such that $\partial \bar X$  is the thick reduction of $\partial X$ in the sense of Scharlau. \newline 
In addition we prove a converse statement saying that an embedding of a thick spherical building at infinity extends to an embedding of the euclidean building having the extended spherical building as its boundary.     
\end{abstract}

\section{Introduction}

In his paper  \cite{Tits1977}  on finite reflection groups that appear as Weyl groups  Tits showed that a thick spherical building never is of type $H_3$ or of any type that has $H_3$ as a standard parabolic.
The natural question then is how the non-thick spherical buildings of these types do arise.
In the same paper Tits gave a construction for such (non-thick) buildings of type $H_3$ via subdivided suspensions of (thick) generalized digons, triangles and pentagons (whose diagrams are the maximal proper sub-diagrams of $H_3$) and $H_4$ via subdivided suspensions of thick buildings whose diagrams are the maximal proper sub-diagrams of $H_4$, and wrongly claimed that all (necessarily non-thick) building of type $H_3$ and $H_4$ arise in this way.

	
Scharlau \cite{Scharlau1987} found a systematic way to describe all those non-thick spherical buildings. The main idea was already proposed by Tits: given a building $\Delta$ one considers equivalence classes of chambers with respect to thin panels.
These thin-classes then form the chambers of a thick spherical building, which was called the \emph{thick frame} of the building $\Delta$ by Scharlau.  We will also refer to the passage from a  building to its thick frame as a \emph{reduction}. This terminology was used by Kleiner and Leeb in Section 3.7 of their work on quasi-isometries between products of symmetric spaces and buildings \cite{KleinerLeeb}. 
	
Scharlau's paper is complemented by Sarah Rees' thesis \cite{Rees1988} which shows how to interpret the statement in the realm of incidence geometries.
Caprace later showed in \cite{Caprace2005} that an analogous result also holds for twin buildings.
Kramer later generalized Scharlau's result and showed in Theorem 3.8 of \cite{Kramer} that the geometric realization of an  $m$-dimensional simplicial sub-complex A of a spherical building which is completely reducible, i.e. in which every vertex has an opposite vertex, is in fact a $k$-fold spherical suspension of a thick spherical building.

One can easily see that such a structure theorem will not hold within the class of simplicial affine buildings. One of the reasons is that some of the spherical Coxeter groups involved do not have a discrete affine counterpart.
One may, however, always consider a non-discrete affine Weyl group which is a semidirect product of a spherical Coxeter group $W_0$ with some $W_0$-invariant translation group $T$.
Replicating the structure theorem to the spherical building at infinity we show that Scharlau's construction canonically extends to euclidean buildings. In fact it does holds for a slightly larger class of generalized affine buildings which allows for non-euclidean metrics when restricted to an apartment. This class is described in \Cref{sec: generalized buildings}.

The precise statements of the main results are as follows.

\begin{mainthm}[Extension]\label{Thm: main thm}
Let $(\bar{X},\bar{\A})$ be a generalized affine building with affine Weyl group $\bar{W}_0 \ltimes \bar{T}$ and suppose that the spherical building $\partial \bar{X}$ of type $\bar{W}_0$ at infinity is thick.   
Then for any spherical Coxeter group $W_0$ of rank $n + k$, with $k\geq 0$, that contains the rank $n$ group $\bar{W}_0$ as a reflection subgroup, there exists a generalized affine building $(X,\A)$ in which $(\bar{X},\bar{\A})$ embeds such that $\partial X$ reduces to $\partial \bar{X}$. I.e. $\partial \bar{X}$ is (as a simplicial complex) isomorphic to the thick frame of the spherical building $\partial X$.  
\end{mainthm}

The building $X$ with atlas $\A$ in the statement of the theorem above is constructed as follows.  
Define $X$ to be the set $\bar X \times \R^k$ equipped with the set  
\[\A := \set{ (f \times \id_{\R^k}) \circ w \mid f \in \bar{\A}, w \in W_T}\] 
of charts with an affine Weyl group defined to be $W_T := W_0 \ltimes T$ for some $W_0$-invariant $\bar{T} \oplus {\set{0_k}} \leq T \leq \R^{n + k}$.
We will prove in \Cref{sec: main proof} that the pair $(X,\A)$ is indeed a generalized affine building with affine Weyl group $W_T$ and that $\partial \bar{X}$ is (as a simplicial complex) 
 isomorphic to the thick frame of the spherical building $\partial X$.
Moreover $\partial X$ is independent of the choice of $T$.

An example to keep in mind is the following.

\begin{Example}
Suppose $X$ is the product of two simplicial regular trees of valency $\geq 3$. Then $\partial X$ carries the structure of an $A_1\times A_1$ spherical building. Each apartment is a 1-sphere subdivided into four 1-simplices. 
The space $X$ itself is a square complex whose links are also spherical buildings of type $A_1\times A_1$. As a reducible building $X$ is thick and its apartments are square complexes whose standard geometric realizations are isomorphic to a copy of the euclidean plane. One can now endow $X$ with the structure of a type $\tilde B_2$ building by subdividing each apartment (in a globally compatible way) using additional diagonal hyperplanes so that the resulting simplicial structure is isomorphic to a Coxeter complex of type $\tilde B_2$.   
\end{Example}

Conversely to the extension theorem above one can show the following reduction theorem.

\begin{mainthm}[Reduction]\label{Thm: splitting}
Let $(X,\A)$ be a generalized affine building with affine Weyl group $W_T = W_0 \ltimes T$ with $W_0$ spherical of rank $n$ and suppose the thick frame $\bar{\D}$ of $\partial X$ is of type $\bar{W}$.
Then $X$ splits as a product $X = \bar{X} \times \R^k$ where $\bar{X}$ is a generalized affine building with $\partial \bar{X} = \bar{\D}$ and $k = \rk(\partial X) - \rk(\bar{\D})$.

If  in addition $\bar{T} := T \cap ( \R^{n-k} \oplus \set{0_k} )$ is $\bar{W}$-invariant, then $\bar{X}$ is a generalized affine building with affine Weyl group $\bar{W}_0 \ltimes \bar{T}$ and $X$ can be constructed from $\bar{X}$ as in \Cref{Thm: main thm}.
\end{mainthm}

One may call $\bar{X}$ the \emph{reduction} (or \emph{soul}, or again \emph{thick frame}) of $X$ having the thick frame of $\partial X$ as its boundary.

\begin{remark}[enrichment of the translation subgroup] \label{Rem: convenient}
	In general, unlike in the spherical case, the constructions in the two main theorems are not mutually inverse,  that is, $(W_0 \ltimes \bar{T}) \cap ( \R^{n - k} \times \set{0_k} ) \neq \bar{T}$.
	The extension construction in most cases substantially enlarges the translation part of the underlying Weyl group.
	This is in particular the case if $k = 0$.
\end{remark}

We carry out the proofs of \Cref{Thm: main thm} and \ref{Thm: splitting} in \Cref{sec: main proof}. They rely on a new axiomatic characterization of (generalized) euclidean buildings due to Curtis Bennet and the first author \cite{Axioms}.
The precise class of generalized affine buildings and their axiomatic description is introduced in \Cref{sec: generalized buildings}, where we also explain  the construction of their spherical building at infinity.
The main feature of this new axiomatic approach is its independence of the choice of a particular Weyl-compatible metric and the fact that the characterization does not involve retractions onto apartments. Both of these features are part of the original definition due to Tits~\cite{Tits1986} as well as the equivalent characterizations due to Parreau~\cite{Parreau}.

It is natural to ask what the relation of the presented reduction theorem to the standard CAT(0) splitting of metric spaces is. This relation is discussed in \Cref{sec: metric pov}.
We proceed below by quickly discussing spherical buildings, their geometric realizations and Scharlau's structure theorem in \Cref{sec: simplicial buildings}.


\section{The reduction theorem for spherical buildings}\label{sec: simplicial buildings}


The main goal of this section is to state and explain Scharlau's reduction theorem for spherical buildings.

\subsection*{Coxeter groups and buildings}

We assume that the reader is familiar with Coxeter systems $(W,S)$ and the structure of their Coxeter complexes $\S(W,S)$.
Put $n=\vert S\vert$ and recall that $W = \lr{S \mid (s_is_j)^{m_{ij}}}$ for a Coxeter matrix $(m_{ij})_{i,j}$ for $1\leq i,j \leq n$.
See Theorem C on page 55 in \cite{Brown} for a proof of the facts provided in \ref{def:reflection action} below. 

\begin{Property}[action on a vector space]\label{def:reflection action}
A Coxeter group $W$ (with the presentation given above) acts by linear transformations on a real vector space $V_W$ with basis $\set{e_1, \dots, e_{n}}$ and preserves the bilinear form $B$ defined by $B(e_i,e_j) := - \cos(\frac{\pi}{m_{ij}})$ if $m_{ij} < \infty$ and $B(e_i,e_j) := -1$ if $m_{ij} = \infty$.
The form $B$ is a scalar product if and only if $W$ is finite.
In this case $W$ acts as a group of orthogonal transformations, and the action is completely determined by its restriction to the unit sphere.
Hence $W$ is in this case called a \emph{spherical} Coxeter group.
\end{Property}

From now on $W$ will always be spherical.

\begin{Property}[cones and the Weyl complex]
The reflection hyperplanes of $W$ in $V$ subdivide $V$ into convex cones.
The set $\S(W,V)$ of cones and their faces is ordered by inclusion and carries the structure of a simplicial complex isomorphic to $\S(W,S)$.
We call $\S(W,V)$ the \emph{Weyl complex} of $(W,S)$ and refer to its elements as \emph{Weyl simplices}.
The intersection with the Weyl simplices induce a simplicial structure on the unit sphere in $V$ called the \emph{geometric realization} of $\S(W,S)$.
Multiplication by $-1$ fixes every reflection hyperplane and hence induces an (involutory) {automorphism} of $\S(W,V)$ or $\S(W,S)$, called the \emph{opposition involution}.
\end{Property}

We are now going to show how inclusions of spherical Coxeter groups determine \emph{geometric inclusions} of their geometric realization. Note that a given finite, real reflection group can be given the structure of  a Coxeter group by choosing a suitable generating set. This choice is not unique. 

\begin{Lemma}[geometric splitting of $V$]\label{geom splitting}
	Let $W$ be spherical, $\rk(W)=n$ and $\bar{W} \leq W$ a reflection subgroup of rank $k$ with generators corresponding to hyperplanes $H_1, \dots, H_k$ in $V$.
	Then $U := \bigcap H_i$ has dimension $n - k \geq 0$ and $V$ splits as the orthogonal sum $V = \bar{V} \oplus U$ with $H_i = \bar{H}_i \oplus U$ for hyperplanes $\bar{H}_i$ in $\bar{V}$.
\end{Lemma}

Using this splitting lemma one can describe what happens on the level of the spherical Coxeter complex, respectively with its geometric realization.

\begin{remark}[direct sums and spherical joins]
Taking products of a (subdivided) vector space with $\R$ yields a spherical suspension of (a subdivided) $\partial V$ at infinity. To provide some more details we use notation from Lemma~\ref{geom splitting}. One can see that the subgroup $\bar{W}$ of $W$ acts as a reflection group on the subspace $\bar{V}$ of $V$. On the contrary we obtain back the walls with respect to $W$ from the walls in $\bar{V}$ with respect to $\bar{W}$ by forming the direct sum $\bar{V} \oplus \R^{n-k}$ and then subdivide by the additional walls seen by $W$ only.

On the level of the unit sphere (or, similarly, the boundary of the vector space) a direct sum with  $\R^{n-k}$ becomes a \emph{spherical join} with $\Sp^{n-k-1}$ in the sense of Definition 5.13 of \cite{BH}. In particular 
\[
\Sp^{n-1} \isom \underbrace{ \Sp^{0} * \dots * \Sp^{0}}_{\substack{\text{$n$ times}}}.
\]
We call $\Sp^{n-1} * \D$ the $n$-fold \emph{suspension} of $\D$. Be aware that the splitting of $\Sp^{n-1}$ into a product of $n$ zero-spheres is not unique.  

\end{remark}

We take the opportunity to remind the reader that the geometric realization of the \emph{simplicial join} of spherical Coxeter complexes is isometric to the \emph{spherical join} of their geometric realizations and that the same is true for spherical buildings, see for example the discussion in Section~2 of \cite{CharneyLytchak2001} or \cite{BH}.

\begin{remark}[inclusion of Coxeter groups]\label{rem:inclusion}
An inclusion of a Coxeter group $\bar{W}$ into a larger Coxeter group $W$ as a reflection subgroup, that is via an injective map $\iota: \bar{W} \to W$ that sends reflections to reflections, is not uniquely determined.
Such a map corresponds to an identification of reflections in $\bar{W}$ with a set of reflection in $W$ which is stable under itself.

The elements in a conjugacy class of the image of any such monomorphism are also possible inclusions of $\bar{W}$ into $W$ preserving reflections, because conjugation preserves the Coxeter relations.

In general, two such embeddings may not be conjugate, for example the embeddings
$A_1 \times A_1 \hookrightarrow C_2$ as long or short roots respectively are not conjugate.

Two examples of such embeddings are shown in \Cref{fig:inclusion}, where the reflection hyperplanes corresponding to the smaller group $\bar{W}$ are shown in bold (red). 
\end{remark}

\begin{figure}[H]
	\centering
	\includegraphics[width=0.35\textwidth]{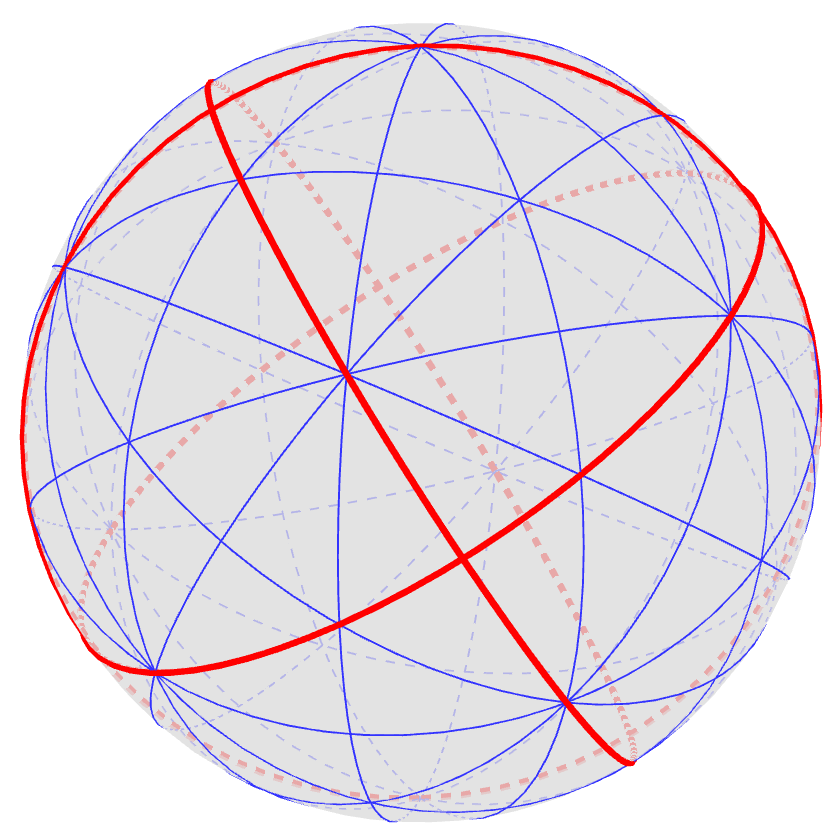}
	\includegraphics[width=0.35\textwidth]{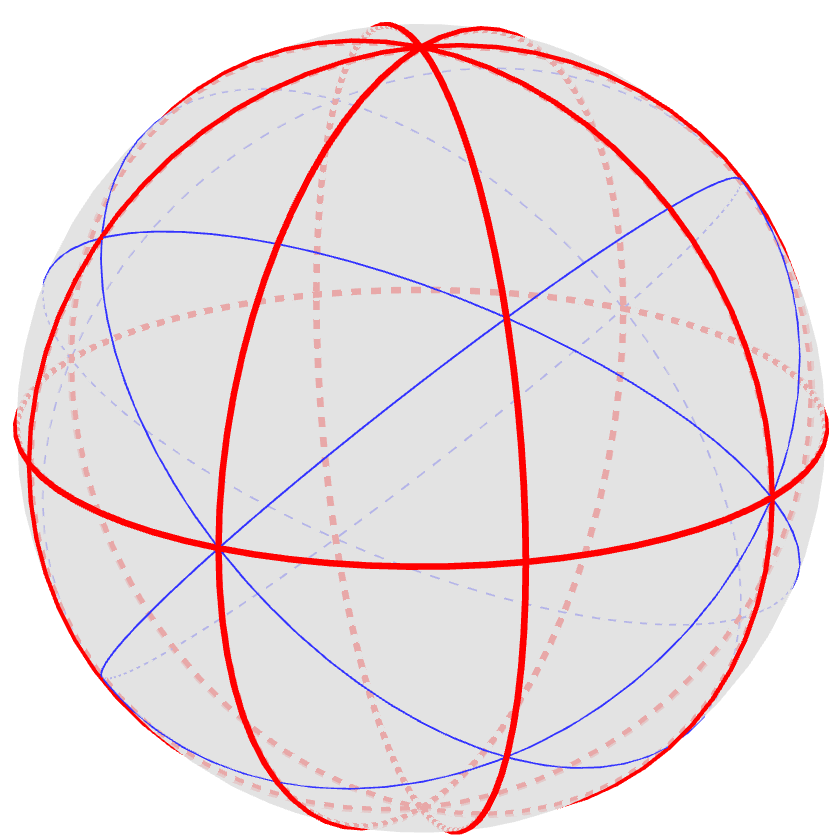}
	\caption{An inclusion of $A_1 \times A_1 \times A_1$ into $H_3$ (left) and an inclusion of $C_2\times A_1$ into $C_3$ (right). The reducible subgroups correspond to the bold red hyperplanes, the ambient groups to bold red and thin blue hyperplanes.  See also  Remark~\ref{rem:inclusion}. }
	\label{fig:inclusion}
\end{figure}

\subsection*{Thick and thin spherical buildings}

We continue by recalling some building terminology. All spherical buildings below are simplicial buildings in the sense of the definition provided in Chapter IV Section 1 of \cite{Brown}. The class of affine buildings we are considering will be  defined in Section~\ref{sec: generalized buildings}.  

In general a (simplicial) building is a simplicial complex $\D$ which is the union of a collection $\A$ of \emph{apartments}, each of which is isomorphic to some (fixed) Coxeter complex. Moreover any building has the properties that every pair of simplices $a,b \in \D$ is contained in a common apartment and that if there are two apartments containing given simplices $a$ and $b$ then there exists an isomorphism between these apartments fixing $a$ and $b$ pointwise.
A building is \emph{spherical} if the corresponding Coxeter group is \emph{spherical}.

It is easy to see that all apartments in a building must be of the same Coxeter type, which we call the \emph{type of the building}. The maximal simplices in a building $\D$ are all of the same dimension and will be called \emph{chambers}, their codimension one faces are called \emph{panels}.

\begin{Definition}[thickness]\label{Def: thick}
A panel $p$ in a building is \emph{thick} if it is contained in at least three chambers. A panel is \emph{thin} if it is contained in exactly two chambers.  A building is \emph{thick} (or \emph{thin}) if all its panels are thick (or thin). 
\end{Definition}

\begin{remark}[weakness vs non-thickness]
Scharlau calls a building weak in \cite{Scharlau1987} if it is not thick. We will simply call a building which does have at least one thin panel a \emph{non-thick} building.

Such a building is then necessarily non-degenerate, because its rank is at least one.

Similarly a building with at least one thick panel is \emph{non-thin}.
\end{remark}

\subsection*{Scharlau's theorem}

We will now explain the main ingredients of Scharlau's proof which are summarized in the following two Lemmata. They allow to distinguish between thick and thin walls in a building and imply that there is a subgroup of the Weyl group generated by reflections in thick walls.

The following is Lemma~1 in \cite{Scharlau1987}. 
\begin{Lemma}[thick and thin walls]\label{Lem: walls}
	Let $\D$ be a spherical building and $M$ a wall in $\D$.
	Then the panels contained in $M$ are either all thick or all thin.
\end{Lemma}

\begin{proof}
	Fix an apartment $\S$ that has $M$ as a wall and let $p,p'$ be distinct panels in $M$.
	Choose a minimal gallery of panels $p = p_0, \dots, p_n = p'$ in $M$.
	This is possible as walls are themselves chamber complexes as e.g. shown in Lemma 4.1 of \cite{Caprace2005}. 
	Consider $a = p \cap p_1$.		
	Since both $p$ and $p_1$ are in $M$ they must be opposite in $\st{\S}{a}$, which is a generalized polygon. A standard result of Tits \cite[Thm. 2.40]{Tits1974} then implies that the stars along the gallery are all isomorphic, i.e.  
	\[
	\st{\D}{p} \isom \st{\D}{p_1} \isom \dots \isom \st{\D}{p'}. 
	\] Hence the lemma follows. 
\end{proof}

Once we have identified thick and thin walls in an apartment we may look at the reflection subgroup of the Weyl group generated by the reflections in thick walls only. In fact the set of reflections in thick walls equals the set of all reflections in the group they generate. So in particular thickness is invariant under reflection. 

This follows from the next lemma, which is Lemma 2 of \cite{Scharlau1987}.
 
\begin{Lemma}[thickness is reflection invariant]\label{Lem: sB thick reflection}
	Let $\S$ be an apartment in a spherical building, let $p$ be a thick panel in $\S$, $\a$ and $-\a$ the half-apartments determined by $p$, and let $\varphi$ be the folding onto $\a$.
	Then $\varphi$ maps thick panels to thick panels and thin panels to thin panels.
\end{Lemma}

We are now going to ``glue'' chambers together to form chambers with only thick panels.

\begin{Definition}[thin-classes]\label{Def: thin-class}
Two chambers $c$ and $c'$ in a spherical building $\D$ are \emph{thin-adjacent} if $c \cap c'$ is a thin panel.
A gallery $\g = (c_0, \dots, c_n)$ is called \emph{thin} if $c_{i-1}$ and $c_{i}$ are thin-adjacent for any $i = 1, \dots, n$, and we say that $c$ and $c'$ are \emph{thin-equivalent} if they can be joined by a thin gallery.
It is easy to see that thin-equivalence is in fact an equivalence relation on $\ch{\D}$, and we refer to the corresponding equivalence classes as \emph{thin-classes}.
For a chamber $c$ we denote its thin-class by $\bar{c}$.
\end{Definition}

One immediately  has the following properties. 
\begin{Lemma}[properties of thin-classes]
For every spherical building the following hold. 
\begin{enumerate}
\item  If two chambers $c$ and $c'$ do not belong to the same thin-class, then every gallery from $c$ to $c'$ crosses a thick wall.
\item  Thin-classes are convex. 
\item  If an apartment $\S$ contains a chamber $c$, then $\bar{c} \subset \S$.
\end{enumerate}
\end{Lemma}

Geometrically the thin-classes are the (closures of the) connected components after having removed the thick walls.
So for every apartment $\S$ in a spherical building $\D$ the group $\bar W$ generated by reflections in thick walls is a Coxeter group (by \Cref{Lem: sB thick reflection}) which acts simply-transitively on the set of thin-classes.
This defines an apartment $\bar{\S}$ by ``forgetting'' some of the structure in $\S$.
Overall we obtain a building $\bar{\D}$ as the union of all apartments $\bar{\S}$, where $\S$ is an apartment in $\D$.
We summarize this in the following Definition/Proposition. 

\begin{defprop}[reduction/thick frame]
For a spherical building $\D$ the set of thin-classes in $\D$ is the set of chambers of a chamber complex $\bar{\D}$ called the \emph{reduction} or \emph{thick frame of $\D$}. By construction $\bar{\D}$ is a union of apartments.  
\end{defprop}

It is easy to see that the thick frame of a thick building is the building itself. 

We are now ready to state Scharlau's structure theorems as shown in \cite{Scharlau1987}. 

\begin{Theorem}[spherical reduction]\label{Thm: Scharlau}
The reduction (or thick frame) $\bar{\D}$  of a spherical building $\D$ of type $W$ is a thick spherical building of type $\bar W$. Two chambers of $\bar{\D}$ are adjacent if and only if they have adjacent representatives in $\D$.
The apartments of $\bar{\D}$ are in one-to-one correspondence with the apartments of $\D$.
The Coxeter group $\bar W$ associated to an apartment $\bar{\S}$ in $\bar \D$  is the subgroup of $W$ generated by the reflections along thick walls of an apartment $\S$ in $\D$, and this way the Weyl group $W(\bar{\D})$ becomes a subgroup of the Weyl group $W(\D)$, which is determined up to conjugation.
\end{Theorem}

\begin{remark}[the degenerate or thin case]\label{rem:deg-spher}
	The thick frame of a thin building i.e. a Coxeter complex is $\{\emptyset\}$.
	This is the Coxeter complex corresponding to the Coxeter system of type $(\{\id\}, \emptyset)$. Its only chamber is $\emptyset$ which has no faces of co-dimension one and is therefore trivially thick (and thin).  This fact was already commented on by Caprace.
	See (1.4) Remark 5. in \cite{Caprace2005}. 
	So we may view a single Coxeter complex of rank $k$ as "split" into a thick building of type $(\{\id \},\emptyset)$ and an $\Sp^{k-1}$.  
\end{remark}


\begin{Theorem}[subdivided suspension]
If $\bar{\D}$ is a thick spherical building of rank $n$ and $W$ some spherical Coxeter group of rank $n +k $ containing $W(\bar{\D})$ as a reflection subgroup, then subdividing the $k$-fold suspension $\Sp^{k - 1} \ast \bar{\D}$ by the walls of $W$ defines a non-thick building $\D$ with thick frame $\bar{\D}$.
This subdivision only depends on the conjugacy-class of the inclusion $W(\bar{\D}) \hookrightarrow W$.
\end{Theorem}

\section{Generalized buildings}\label{sec: generalized buildings}

In this section we introduce the class of generalized affine buildings for which we will prove a natural extension of Scharlau's construction.  

We follow \cite{Axioms} but will only consider generalized affine buildings with apartments isomorphic to $\R^n$ and start with a definition of  generalized Weyl groups that will be considered on $\R^n$. 

\begin{Definition}[(generalized) affine Weyl group]\label{Def: affine Weyl group}
Let $W_0$ be a spherical Coxeter group and denote by $\RS$ its root system. 
Define $\Aff$ to be the $\R$-span of $\RS$. Then $W_0$ acts on $\Aff$ by naturally extending its action on $\RS$. 
Suppose $T$ is a $W_0$-invariant translation subgroup of the automorphism group of the affine space underlying $\Aff$. Then $W_T := W_0 \ltimes T$ is called an \emph{affine Weyl group} and the images of Weyl simplices under $W_T$ are called \emph{(affine) Weyl simplices}.
The top-dimensional affine Weyl simplices will be called \emph{Weyl chambers}.
\end{Definition}

An element of $W_T$ is an \emph{affine reflection} if it can be written as a pair of a translation $t$ and a some reflection $r_\alpha$ in $W_0$. For an explicit description of affine reflections see Definition 4.7 and Remark 4.8 in \cite{Thesis}. 

The fixed point set $H_{w}$ of an affine reflection $w$ is called an \emph{ affine hyperplane}. Any affine hyperplane  splits $\Aff$ into two half spaces. We say that $H_{w}$ or $w$ \emph{separates} points $x$ and $y$ in $\Aff$ if $x$ and $y$ are contained in different open half spaces with respect to $H_{w}$.

Affine Weyl simplices are the translates of Weyl simplices under $T$ and thus are closed, convex subsets of $V$.
If $W_0$ is of rank $n$, then $V \isom \R^n$.

We define below the notion of a Weyl--compatible metric which will be used to define a neighborhood in \Cref{Def: germ}.

\begin{Definition}[Weyl--compatible metric]\label{Def: Weyl compatible metric}
	A metric $d$ on $\Aff$ is called \emph{Weyl--compatible} if it is $W_T$-invariant and satisfies the following conditions
	\begin{labeling}{loong}
	\item[(1)] Let $S$ be any Weyl chamber and let $P,Q$ be Weyl simplices of the same type with sub-faces $P' \fo P$ and $Q' \fo Q$. If there exists a uniform constant $\lambda > 0$ (independent of $S,P,Q, P', Q'$) such that for $y \in Q'$ there exists $x \in P'$ with $|d(x,v)-d(y,v)| \leq \lambda$ for all $v \in S$, then $P$ is a translate of $Q$.
	\item[(2)] Whenever a hyperplane $H$ separates $x$ and $y$ in $\Aff$ there is a $z \in H$ such that $d(x,y) = d(x,z) + d(z,y)$.
	\end{labeling}
\end{Definition}

For a reformulation of \ref{Def: Weyl compatible metric} (1) see also the first item of Definition 3.1 in \cite{Axioms}. 
Weyl--compatibility ensures that two sector panels of the same type are parallel if and only if they are translates of one another.


We follow \cite{Axioms} and define a class of generalized affine buildings as spaces constructed from a model space $\A$ which satisfy some extra conditions. It was shown in \cite{Axioms} Chapter 10 that the euclidean metric on $\R^n$ is Weyl--compatible for any $n$. So the class of generalized affine buildings defined below includes all euclidean ($\R$-) buildings in the sense of Parreau~\cite{Parreau} and allows in addition for buildings with metrics induced by more general metrics on the apartments.

\begin{Definition}[spaces modeled on $\Aff$]\label{Def: spaces modeled on A}
Let $W_0, W_T$ and $\Aff$ as in \Cref{Def: affine Weyl group}. A \emph{space modeled on $\Aff$} is a pair $(X,\A)$ with $X$ any  set together with an \emph{atlas} $\A$, that is a family of injective \emph{charts} $f \colon \Aff \hookrightarrow X$, satisfying the following axioms:

\begin{labeling}{loong}
\item[{(A1)}] For any $f \in \A$ and $w \in W_T$, $f \circ w \in \A$.
\item[{(A2)}] For any $f,g \in \A$ with $f(\Aff) \cap g(\Aff) \neq \emptyset$, the preimage  $f^{-1}(g(\Aff))$ is closed and convex in $\Aff$ and $\restr{f}{f^{-1}(g(\Aff))} = \restr{(g \circ w)}{f^{-1}(g(\Aff))}$ for some $w \in W_T$.
\item[{(A3)}] For any pair of points  $x,y \in X$ there exists $f \in \A$ such that $x,y \in f(\Aff)$.
\end{labeling}
The sets $f(\Aff)$, for $f\in\A$ are called \emph{apartments} of $X$. 
\end{Definition}

In particular the axioms imply that $X$ is as a set the union of its apartments and transition maps are given by elements of $W_T$. Therefore the notion of Weyl chambers and their germs and types makes sense for $X$.


\begin{Definition}[germs of Weyl simplices]\label{Def: germ}
	Let $P,Q$ be Weyl simplices in $\Aff$ or in $X$ both based at some point $x$.
	We say that $P$ and $Q$ \emph{share the same germ} if $P \cap Q$ is a neighborhood of $x$ in $P$ and $Q$.
	Moreover, $(P \cap Q)\setminus \set{x} \neq \emptyset$.\newline 	
	This is an equivalence relation on the set of $x$-based Weyl simplices, and we denote the equivalence class of $P$, called the \emph{germ of $P$ in $x$} by $P_x$.
	The set $\D_x$ of all $x$-based germs is ordered by inclusion.\newline
\end{Definition}

In an apartment $f(\Aff)$, the set of $x$-based Weyl simplices is isomorphic to $\S(W_0,S)$ and two $x$-based germs are \emph{opposite} if they are contained in a common apartment and are images of one another under the opposition involution. For Weyl chambers this means they are images of one another under the (unique) longest element $w_0$ of the spherical Weyl group. Two Weyl chambers (or simplices) are \emph{opposite at $x$} if their germs are opposite.	

We say that a germ $S_x$ of a Weyl chamber $S$ based at $x$ is \emph{contained in an apartment $A$} if and only if there exists some $\varepsilon >0$ such that 
\[
S\cap (B_\varepsilon(x)\setminus{\{x\}}) \subset A, 
\]
where $B_\varepsilon(x)$ denotes the $\varepsilon$-ball around $x$. 

We may now define generalized affine buildings. 

\begin{Definition}[generalized affine buildings]\label{Def: generalized building}
A space modeled on $\Aff$ is a \emph{generalized affine building} if the following two axioms are satisfied. 
\begin{labeling}{loong}
\item[{(GG)}] Any two germs of Weyl chambers based at the same vertex are contained in a common apartment.
\item[{(CO)}] Any two opposite Weyl chambers based at the same vertex are contained in a unique apartment.
\end{labeling}
\end{Definition}

\begin{remark}[geometric meaning of a germ]
One can think of the germs in $x$ as the set of local (initial) directions at $x$.
Then {(GG)} essentially says that any pair of possible combinatorial directions is represented in some apartment. In fact the set of germs of Weyl chambers based at a point $x$ form the set of chambers of a spherical building $\D_x$, the \emph{building of germs at $x$}. This is is the analogue of a star or link of a vertex in the simplicial setting. For details see Theorem 5.17 in~\cite{Thesis}. 
\end{remark}

\begin{remark}[axiom (A4)]
	If $X$ is a generalized affine building, then by the equivalences shown in \cite{Axioms}, also the following property holds:
\begin{labeling}{loong}
\item[(A4)] 	Given two Weyl chambers, there is a pair of sub-Weyl chambers contained in a common apartment.
\end{labeling}
This axiom is part of the original definition in \cite{Tits1986}.
\end{remark}

\subsection*{Spherical building at infinity}

We spend the rest of this section to explain that the structure at infinity of a generalized affine building is a simplicial spherical building. Proofs can be found in \cite{Axioms} Chapter 9.

\begin{Definition}[parallel Weyl simplices]\label{Def: parallel simplices}
	Let $(X, \A)$ be a generalized affine building.
	Two Weyl simplices $P$ and $Q$ are \emph{parallel} if there is a sequence $P = P_0, P_1, \dots, P_n = Q$ such that $P_{i-1}$ and $P_{i}$ are translates of each other in some apartment for $i = 1, \dots, n$.
\end{Definition}

In the light of $(A4)$ this means that two Weyl chambers are parallel if and only if they have common sub-Weyl chamber and that such a sequence can always be chosen with $n \leq 3$.

Note that Weyl simplices are defined with respect to a certain (fixed) atlas $\A$ of $X$, and so is parallelism.
It is easy to see that parallelism is an equivalence relation and we denote the parallel-class of a Weyl simplex $P$ by $\partial_\A P$ or simply $\partial P$ and the set of all parallel classes of Weyl simplices in $X$ by $\partial_\A X$ or $\partial X$.

\begin{remark}[parallelism in case of metric buildings]
In case the building $X$ is equipped with a Weyl--compatible metric $d$ one can show that two Weyl simplices are parallel if and only if they have finite Hausdorff-distance with respect to $d$.
\end{remark}

For a proof of the following theorem see \cite[Prop. 5.7]{Thesis} or \cite{Bennett90}. 

\begin{Theorem}[the spherical building at infinity]\label{Thm: sB at infinity}
Let $(X,\A)$ be a generalized affine building.
The set of parallel classes of Weyl simplices  $\partial_\A X$ of a generalized affine building $(X, \A)$ of type $W_T=W_0 \ltimes T$ is a spherical (simplicial) building of type  $W_0$. We call $\partial_\A X$ the \emph{spherical building at infinity} of $X$. The apartments in $\partial X$ are in bijection with the apartments in $X$.
\end{Theorem}

Note that $\partial_\A X$ does depend on the choice of $\A$, but not on $T$. Sometimes $X$ may carry more than one possible atlas, say $\A\neq\A'$, and in this case $\partial_\A X\neq \partial_{\A'} X$ may hold.

\section{Proofs of \Cref{Thm: main thm} and \Cref{Thm: splitting}}\label{sec: main proof}

This section contains the proofs of our main theorems.  
We will be working with generalized affine buildings whose apartments are copies of $\R^n$ for some $n$.
However, we do not require them to be equipped with the euclidean metric.

We omit proofs for the part about $\partial X$, as they follow directly from \Cref{Thm: sB at infinity}.

\begin{remark}[the degenerate case]
We have already discussed the degenerate spherical case in Remark~\ref{rem:deg-spher}.
Observe that the affine versions of the splitting theorem also holds in the degenerate (affine) case.
Here a single flat of dimension $k$ ``splits" as a product of a building of trivial type, with one apartment isomorphic to $R^0$, having a thick building of type $(\{\id\}, \emptyset)$ at infinity  and an euclidean factor of dimension $k$. Here the building at infinity exists only in a combinatorial sense and not metrically. This mimics the corresponding ``splitting" into $\{\emptyset\}$ and an $\Sp^{k-1}$ we have discussed in \ref{rem:deg-spher}.   	
\end{remark}

Let us recall the statement of our first main theorem. 

\begin{mainthmA}[Extension]
	Let $(\bar{X},\bar{\A})$ be a generalized affine building with affine Weyl group $\bar{W}_0 \ltimes \bar{T}$ and suppose that the spherical building $\partial \bar{X}$ of type $\bar{W}_0$ at infinity is thick.   
	Then for any spherical Coxeter group $W_0$ of rank $n + k$, with $k\geq 0$, that contains the rank $n$ group $\bar{W}_0$ as a reflection subgroup, there exists a generalized affine building $(X,\A)$ in which $(\bar{X},\bar{\A})$ embeds such that $\partial X$ reduces to $\partial \bar{X}$. I.e. $\partial \bar{X}$ is (as a simplicial complex) isomorphic to the thick frame of the spherical building $\partial X$.  
\end{mainthmA}
\begin{proof}[Proof of \Cref{Thm: main thm}]
	
By construction, apartments in $X$ are of the form $f(\bar{\Aff}) \times \R^k = (f \times \id_{\R^k})(\Aff)$ with $\Aff$ the affine space associated with $\R^{n + k} \isom V_{W}$.
Since the elements of $\bar{W}_0$ correspond by construction to the elements $w \times \id_{\R^{k}}$ in $W_0$, axioms {(A1)} and {(A3)} are satisfied.

In particular $\A$ is invariant under pre-composition of its elements by the elements of $W_T$ by construction.
We can therefore without loss of generality assume that the elements of $\A$ are of the form $f = \bar{f} \times \id_{\R^k}$ by pre-composing with suitable elements of $W_T$, for example the inverse of a representative of the coset $(W_0 \times \id_{\R^{k}})(w,t)$.
		
Suppose $C := f(\Aff) \cap g(\Aff) \neq \emptyset$.
By pre-composing suitable (and possibly different) elements of $W_T$ we can assume $f = \bar{f} \times \id_{\R^k}$ and $g = \bar{g} \times \id_{\R^k}$ thus $C$ splits as a product $C = \bar{C} \times \R^k$, with $\bar{C}$ some non-empty and thus closed convex set.
Hence $\restr{f}{\bar{C}}$ and $\restr{g}{\bar{C}}$ differ by some element of $\bar{W}$ and thus axiom {(A2)} follows.

By construction a Weyl chamber in $X$ is a subset of $\bar{S} \times \R^k$ where $\bar{S}$ is a Weyl chamber  in $\bar{X}$.
Suppose $S$ and $T$ are opposite Weyl chambers based at the same vertex in $X$, arising from Weyl chambers $\bar{S}$ and $\bar{T}$ in $\bar{X}$.
We can assume that $\bar{S}$ and $\bar{T}$ are based at the same vertex, thus they determine a unique apartment $f(\Aff)$ in $\bar{X}$, and $f(\Aff) \times \R^k$ is the unique apartment in $X$ containing $S$ and $T$, hence {(CO)}.

Similarly if $S_x$ and $T_x$ are the germs of Weyl chambers $S$ and $T$ based at the same vertex in $X$, arising from Weyl chambers $\bar{S}$ and $\bar{T}$ in $\bar{X}$, we can assume they are based at the same vertex, say $\bar{x}$.
Then $\bar{S}_{\bar{x}}$ and $\bar{T}_{\bar{x}}$ are contained in some apartment $f(\Aff)$ in $X$, hence {(GG)}.
\end{proof}

Our second main result is the following theorem. 

\begin{mainthmB}[Reduction]
	Let $(X,\A)$ be a generalized affine building with affine Weyl group $W_T = W_0 \ltimes T$ with $W_0$ spherical of rank $n$ and suppose the thick frame $\bar{\D}$ of $\partial X$ is of type $\bar{W}$.
	Then $X$ splits as a product $X = \bar{X} \times \R^k$ where $\bar{X}$ is a generalized affine building with $\partial \bar{X} = \bar{\D}$ and $k = \rk(\partial X) - \rk(\bar{\D})$.
	
	If  in addition $\bar{T} := T \cap ( \R^{n-k} \oplus \set{0_k} )$ is $\bar{W}$-invariant, then $\bar{X}$ is a generalized affine building with affine Weyl group $\bar{W}_0 \ltimes \bar{T}$ and $X$ can be constructed from $\bar{X}$ as in \Cref{Thm: main thm}.
\end{mainthmB}

\begin{proof}[Proof of \Cref{Thm: splitting}]
Let $S$ and $T$ be representatives of chambers $\partial S$ and $\partial T$ respectively.
Then by axiom (A4) there is an apartment $f(\Aff)$ containing sub-Weyl chambers $S'$ and $T'$ of $S$ and $T$ respectively, and by applying some translation we may without loss of generality assume they are based at the same vertex.
In particular if $\partial S$ and $\partial T$ are opposite chambers in $\partial X$, they determine a unique apartment in $X$, which proves the last part of \Cref{Thm: sB at infinity}.

Similarly if two apartments $f(\Aff)$ and $g(\Aff)$ contain representatives $S_f$ and $S_g$ of a chamber $\partial S$, there is a common sub-Weyl chamber $S \subset f(\Aff) \cap g(\Aff)$, and we may by axiom (A1) and the same argument as in the proof of Theorem A
without loss of generality assume that $f$ and $g$ agree on $f^{-1}(S)$.
If on the other hand $\S$ and $\S'$ are apartments in $\partial X$, then the isomorphism in axiom  {(B2)} in the definition of a (spherical) building can by chosen to be type-preserving, i.e. an element of $W_0$, and since translations are invisible to $\partial X$ (or parallelism) such a splitting must occur apartment-wise.

But in any apartment this is the splitting from \Cref{geom splitting}.
Assuming $T = \R^n$, $\bar{X}$ clearly has the structure of a generalized affine building, that is $f$ and $g$ are of the form $f = \bar{f} \times \id_{\R^{k}}$ and $g = \bar{g} \times \id_{\R^{k}}$, and for arbitrary $T$ our construction for $\bar{T}$ ensures $W\bar{T} \subset T$, which is a sufficient condition to apply the construction of \Cref{Thm: main thm}.
\end{proof}

\section{A metric point of view}\label{sec: metric pov}

In this section we will discuss the connection with the CAT(0) splitting theorem. As before the generalized affine buildings have apartments that are copies of $\R^n$ for some $n \geq 0$ equipped with an arbitrary Weyl--compatible metric $d$ defined in \Cref{Def: Weyl compatible metric}.
Any Weyl--compatible metric $d$ on the model space or an apartment extends to a metric on $X$, which we also denote by $d$.

\begin{Example}[generalized buildings and the CAT(0) property]
Generalized affine buildings are in general not CAT(0) spaces.
One can see that the CAT(0) property fails already on an apartment level. For example, the \emph{maximum norm} $|\cdot|_\infty$ induces a $A_1 \times A_1$ or $C_2$ - invariant (and compatible) metric on $\R^2$ which is not uniquely geodesic and hence not CAT(0).
\end{Example}

If $d$ is the euclidean metric, then Charney and Lytchak show in Proposition 2.3 of \cite{CharneyLytchak2001} that generalized affine buildings are CAT(0) spaces that have the geodesic extension property. 

We will therefore discuss in this section the metric implications of our main theorem in connection with Bridson's CAT(0)-splitting theorem, see Theorem 9.24.  in \cite{BH}. Note that CAT(0)-splittings of Euclidean buildings are also discussed by Kleiner and Leeb in Section 4.3 of \cite{KleinerLeeb}. 

\begin{Theorem}[CAT(0) splitting]
	Let $X$ be a complete CAT(0) space in which all geodesic segments can be extended to (bi-infinite) geodesic lines.
	If $\partial_T X$ is isometric to the spherical join of two non-empty spaces $A_1$ and $A_2$, then $X$ splits as a product $X = X_1 \times X_2$ where $\partial_T X_i = A_i$ for $i = 1,2$.
\end{Theorem}

As pointed out in 6.1 of \cite{CharneyLytchak2001} thin panels of a spherical building are invisible to the metric.
Hence a non-thick spherical building is indistinguishable from the $n$-fold suspension of a thick spherical building.
In the sense of \Cref{Thm: Scharlau}, the non-thick buildings come from some inclusion $W \hookrightarrow W \times (\mod{\Z}{2 \Z})^n$.

The CAT(0) splitting theorem readily allows us the identify $\E^n$ as one of the splitting factors in \Cref{Thm: splitting}, but a priori it is not clear that the second factor carries the structure of a generalized affine building, or even of a space modeled on $\Aff$.
In order to establish such a characterization one necessarily needs at least to find a one-to-one correspondence of apartments as done in the proof of \Cref{Thm: splitting} -- at which point CAT(0) becomes superfluous as the actual metric on the space has no influence on the proof of \Cref{Thm: splitting}.

However the CAT(0) case imposes further restriction on the metric. The metric $d$ necessarily is the product metric $d = \sqrt{d_{\bar{X}}^2+d_\E^2}$ metric, making $\bar{X}$ a convex subset of $X$.


\end{document}